\documentclass{amsart}
\usepackage{amsmath}
\usepackage{amsthm}
\usepackage{amsfonts}
\usepackage{mathdots}
\usepackage{graphicx}
\usepackage{latexsym}
\usepackage[thinlines]{easybmat}

\newcommand{\R}{\ensuremath{ {\mathbb R} }}
\newcommand{\Z}{\ensuremath{ {\mathbb Z} }}

\newcommand{\bX}{\ensuremath{ \bar{X}} }

\newcommand{\bt}{\ensuremath{ \bar{\theta}} }
\newcommand{\nab}[3]{\ensuremath{ \left< \nabla_{\bX_{#1}}\bX_{#2},\bX_{#3} \right>}}
\newcommand{\na}[1]{\ensuremath{ \nabla_{\bX_{#1}}}}
\newcommand{\koszr}[3]{\ensuremath{\left< \left[ #1,#2 \right],#3 \right>-\left< \left[ #2,#3 \right],#1 \right> + \left< \left[ #3,#1 \right],#2 \right> }}

\newcommand{\Rm}[3]{\ensuremath{\nabla_{#1}\nabla_{#2}#3 - \nabla_{#2}\nabla_{#1}#3 - \nabla_{\left[ #1,#2 \right]}#3}}

\DeclareMathOperator{\tr}{tr}

\newtheorem{thm}{Theorem}
\newtheorem{prop}[thm]{Proposition}
\newtheorem{lem}[thm]{Lemma}

\theoremstyle{definition}
\newtheorem{defn}{Definition}
\def\urltilde{\kern -.15em\lower .7ex\hbox{\~{}}\kern .04em}

\author{Christopher Godbout}
\title{Chern-Simons classes and the Ricci flow on 3-manifolds}
\thanks{This work is part of the author's dissertation at Lehigh University. The author wishes to thank his advisor, David Johnson, for his help and insight.}
\date{\today}

\begin{document}
	\begin{abstract}
		In 1974, S.-S. Chern and J. Simons published a paper where they defined a new type of characteristic class - one that depends not just on the topology of a manifold but also on the geometry. The goal of this paper is to investigate what kinds of geometric information is contained in these classes by studying their behavior under the Ricci flow. In particular, it is shown that the Chern-Simons class corresponding to the first Pontryagin class is invariant under the Ricci flow on the warped products $S^2\times_f S^1$ and $S^1 \times_f S^2$ but that this class is not invariant under the Ricci flow on a generalized Berger sphere. 
	\end{abstract}
	\maketitle
	\nocite{Scott1983}

	\section{Introduction}
	cIn 1974, S.-S. Chern and J. Simons introduced a collection of secondary characteristic classes on a Riemannian manifold, $M$, which reflect some of the geometric structure of the manifold in addition to its topological structure. In particular, if $M$ is a compact 3-manifold,  the secondary characteristic class $\left\{\widetilde{TP_1\left(\Omega\right)}\right\}\in H^3\left( M,\R/\Z \right)$ is invariant under a conformal change in the metric. So metrics with different Chern-Simons classes will have different conformal structures. However, it is still unclear in general which geometric properties of a Riemannian manifold are reflected in these classes.

	The present work is a step in seeing what sort of geometric information can be determined by these classes. For a 3-manifold, the Ricci curvature flow deforms the metric of the manifold towards the standard geometries of Thurston's geometrization program. If it were true that the Chern-Simons class of 3-manifolds (specifically corresponding to the first Pontryagin form) were invariant under the Ricci flow, then perhaps these classes would reflect something of this geometric structure. See \cite{Scott1983} for a detailed treatment on the different geometries.

	However, it will be shown that this is not so in general. For a warped product of the form $S^2 \times_f S^1$, the Chern-Simons class is invariant as this flows under the Ricci flow towards the geometry of the product $S^2 \times S^1$. However, for a generalized Berger sphere, which for generic choices would flow towards a spherical geometry, the Chern-Simons class is not invariant.

	Let $G$ be a Lie group with Lie algebra $\mathfrak g$. Let $M^n$ be a smooth $n$-dimensional manifold with $E\to M$ a principal $G$-bundle. Denote its connection and curvature forms by $\omega$ and $\Omega$, respectively. Additionally, let $\left<\theta^1,\theta^2,\dots,\theta^n\right>$ be an orthonormal coframe.

	Then
	\begin{align}
		R_{ijk\ell} =& g_{\ell m}{R_{ijk}}^m\\
		\omega_i^j =& \Gamma_{ik}^j \theta^k\\
		\Omega_i^j =& \frac{1}{2} {R_{k\ell i}}^j \theta^i \wedge \theta^j
	\end{align}

	\begin{defn}
		\label{dfn:cspoly}
		A \emph{polynomial of degree $\ell$}, $P$, is a symmetric and multilinear map $P: {\mathfrak g}^\ell \to \R$. The Lie group $G$ acts on $\mathfrak g$ by inner automorphism. If the polynomial is invariant under this action, then it is said to be an \emph{invariant polynomial}. The set of invariant polynomials of degree $\ell$ is denoted $I^\ell(G)$.
	\end{defn}

	For $G = Gl_n(\R)$, the first Pontryagin polynomial is denoted $P_1 \in I^2(G)$ and is defined by

	\begin{equation}
		\label{eqn:p1}
		P_1(A\otimes B) = \frac{1}{2\pi^2}\left( \tr A \tr B - \tr \left( AB \right)  \right)
	\end{equation}

	\begin{defn}
		\label{dfn:csform}
		For $P\in I^\ell(G)$, define the 2-form $\varphi_t = t\Omega + \frac{1}{2}\left( t^2-t \right)\left[ \omega,\omega \right]$. The \emph{Chern-Simons form of P} is defined to be 
		\begin{equation}
			\label{eqn:TP}
			TP\left( \omega \right) = \ell\int_0^1 P\left( \omega\wedge\varphi_t^{\ell-1} \right)dt.
		\end{equation}
		Since $\varphi_t$ is a 2-form, order is irrelevant. The form $TP\left( \omega \right)$ is a $\left(2\ell-1\right)$-form on $E$ and Chern and Simons showed in \cite[Prop. 3.2]{Chern1974} that $dTP\left(\omega\right)=P\left(\Omega\right)$.
	\end{defn}

	Because of this, if $P\left(\Omega\right) = 0$, then the form is closed and so defines a cohomology class in $H^{2\ell-1}\left(E,\R\right)$. Denote this class $\left\{TP\left(\omega\right)\right\}$. Theorem 3.9 of \cite{Chern1974} shows that if $\dim M = n$ and $2\ell-1 = n$, then $TP\left(\omega\right)$ is closed and $\left\{TP\left(\omega\right)\right\}\in H^n\left(E,\R\right)$ depends on $\omega$. In this paper,  it will usually be the case that $\ell = 2$ and so $2\ell-1=3=\dim M$. Let $\sim$ represent the reduction of a cohomology class mod $\Z$. Theorem 3.16 of \cite{Chern1974} shows that there is a well-defined cohomology class $\left\{\widetilde{TP\left(\omega\right)}\right\}\in H^{2\ell-1} \left(M, \R/\Z\right)$ whose lift is $\left\{\widetilde{TP(\left(\omega\right)}\right\}\in H^{2\ell-1} \left(E(M),\R/\Z\right)$.

	The following result is proved in \cite[Prop 3.8]{Chern1974}.

	\begin{prop}
		\label{prop:tpd}
		Let $\omega(s)$ be a smooth 1-parameter family of connections on $E\to M$ with $s\in \left[ 0,1 \right]$. Set $\omega = \omega(0)$ and $\dot{\omega} = \dfrac{d}{ds} \omega(s) \mid_{s=0}$. For $P\in I^\ell(G)$,
		\begin{equation}
			\label{eqn:tpd}
			\dfrac{d}{ds}TP\left( \omega(s) \right)\mid_{s=0} = \ell P\left( \dot{\omega} \wedge \Omega^{\ell-1} \right) + exact.
		\end{equation}
	\end{prop}

	Given a 1-parameter family of metrics, $g(t)$, on a Riemannian $M^n$, defined on some interval $I\subset \R$, Richard Hamilton , in \cite[p. 259]{Hamilton1982}, defined the Ricci flow equation as
	\begin{equation}
		\left\{\begin{array}{rl}
				\frac{\partial}{\partial t} g_{ij}(t) &= -2R_{ij}(t) \\
				g(0) &= g_0.
			\end{array}\right.
		\end{equation}

		Hamilton proved the following important theorem in corollaries 17.10 and 17.11 of  \cite{Hamilton1982}.
		\begin{thm}[Hamilton, 1982]
			Let $\left(M^3,g_0\right)$ be a closed Riemannian 3-manifold with positive Ricci curvature. Then there exists a unique solution, $g(t)$, of the normalized Ricci flow with $g(0)=g_0$ for all $t\geq 0$. Furthermore, as $t\to\infty$, the metrics $g(t)$ converge exponentially fast in every $C^m$-norm to a $C^\infty$ metric $g_\infty$ with constant positive sectional curvature.
		\end{thm}

		An excellent, in-depth treatment of the Ricci flow is \cite{Chow2006}.

		\section{Warped Products of Spheres}

		The simplest example of how the Chern-Simons classes are related to the Ricci flow is given by the round $n$-sphere. 

		\begin{prop}
			Let $M=S^{2\ell-1}$ with the standard metric and $P\in I^{2\ell-1}\left( Gl_n\left( \R \right) \right)$ be any invariant polynomial. Then the Chern-Simons class,
			\begin{equation*}
				\left\{ \widetilde{TP\left( \omega \right)} \right\}\in H^{2\ell-1} \left( M,\R/\Z \right),
			\end{equation*}is invariant under the Ricci flow.
		\end{prop}

		\begin{proof}
			The Ricci flow causes the sphere to shrink to a point in finite time. At every time, however, it is still a round sphere and so the Ricci flow is a conformal change in metric. Since the Chern-Simons classes are conformal invariants, they are therefore invariant under the Ricci flow.
		\end{proof}

		\begin{defn}
			Given manifolds $\left( M^m,g_M \right)$ and $\left( N^n,g_N \right)$ along with a smooth function $f: M \to \R^+$, the \emph{warped product} of $M$ with $N$, denoted $M \times_f N$ is defined by $\left( M\times N, g_M \oplus f g_N \right)$. The function $f$ will be referred to as the \emph{warping function}.
		\end{defn}

		Let $\left( M,g \right)$ be the warped product of $S^n$ with $S^m$ with $m+n=3$ and using the standard round metrics. Let $f:S^n\to \R^+$ be the warping function. Let $\left\{ \theta_i \right\}_{i=1}^n$ be the coordinates for $S^n$ and $\left\{ \theta_i \right\}_{i=n+1}^{m}$ be the coordinates for $S^m$. These are the standard spherical coordinates. For $S^1$,
		\begin{equation*}
			x_1 = \cos\theta^1, x_2 = \sin\theta^1.
		\end{equation*} For $S^2$,
		\begin{equation*}
			x_1 = \cos\theta^1, x_2 = \sin\theta^1\cos\theta^2, x_3 = \sin\theta^1 \sin\theta^2.
		\end{equation*}
		Thus $g$ is diagonal and
		\begin{equation}
			g_{ii} = \left\{ \begin{array}{rl}
					\prod_{j=1}^{i-1}\sin^2\theta^j& 1 \leq i \leq n\\
					f\prod_{j=n+1}^{i-1}\sin^2\theta^j & n+1 \leq i \leq n+m
				\end{array}\right.
			\end{equation}
			Formulas for the Christoffel symbols and curvature of the warped products can be found in \cite[pp. 209-211]{Oneill1983}.

			\begin{lem}
				\label{lem:0deriv}
				Given $n+m=3$ and $f:S^n \to \R^+$, let $M = S^n\times_f S^m$. For the first Pontryagin polynomial $P_1\in I^2\left( Gl_n\left( \R \right) \right)$, the derivative of $TP_1\left( \omega \right)$ is exact at time 0.
			\end{lem}

			\begin{proof}
				For $n=1$, the Christoffel symbols are computed using the standard formula $\Gamma_{ij}^k = \frac12 g^{k\ell}\left(\partial_i g_{j\ell} + \partial_j g_{i\ell} - \partial_\ell g_{ij}\right)$. From there, the connection 1-forms are computed using the formula $\omega_i^j = \Gamma_{ki}^jd\theta^k$. Finally, the curvature 2-forms are computed according to $\Omega_i^j = d\omega_i^j - \omega_i^p\wedge\omega_p^j$. 

				Based off of these computations along with the fact that $\Omega_i^j = {R_{pqi}}^j \theta^p \wedge \theta^q$,
				\begin{equation*}
					\Omega = \left(\begin{array}{ccc}
							0 & {R_{121}}^2 d\theta^1 \wedge d\theta^2 & {R_{131}}^3 d\theta^1 \wedge d\theta^3 \\
							{R_{122}}^1 d\theta^1 \wedge d\theta^2 & 0 & {R_{232}}^3 d\theta^2 \wedge d\theta^3\\
							{R_{133}}^1 d\theta^1 \wedge d\theta^3 & {R_{233}}^2 d\theta^2 \wedge d\theta^3 & 0
						\end{array}\right).
				\end{equation*}

				The variational formulas for the Christoffel symbols under the Ricci flow are
				\begin{equation}
					\label{eq:chrd}
					\partial_t \Gamma_{ij}^k = \frac{R_{i\ell}}{g_{ii}g_{\ell\ell}}\left(\partial_i g_{j\ell} + \partial_j g_{i\ell} - \partial_\ell g_{ij}\right) - g^{k\ell} \left(\partial_i R_{j\ell} + \partial_j R_{i\ell} - \partial_\ell R_{ii}\right)
				\end{equation}
				Using \eqref{eq:chrd} combined with the formula for $\Omega$, the derivative of $\omega$ is of the form
				\begin{equation*}
					\dot\omega = \left(\begin{array}{ccc}
							a_1^1 d\theta^1 & a_1^2 d\theta^2 & a_1^3 d\theta^3\\
							a_2^1 d\theta^2  & a_2^2 d\theta^1 & 0\\
							a_3^1 d\theta^1 & 0 & a_3^3 d\theta^3
						\end{array}\right).
				\end{equation*} Since the curvature form $\Omega$ has zeroes along the diagonal, $\tr\Omega=0$ and thus $\tr \dot{\omega} \tr\Omega =0$. The product $\dot{\omega}\wedge \Omega$ has the following diagonal

				\begin{align*}
					\left(\dot\omega \wedge \Omega\right)_1^1 =& \left(a_1^2 d\theta^2 \right) \wedge \left({R_{122}}^2 d\theta^1 \wedge d\theta^2\right) + \left( a_1^3 d\theta^3 \right) \wedge \left({R_{133}}^3 d\theta^1 d\theta^3 \right)\\
					=& 0\\
					\left(\dot\omega \wedge \Omega\right)_2^2 =& \left(a_2^1 d\theta^2\right) \wedge \left({R_{121}}^2 d\theta^1 \wedge d\theta^2\right)\\
					=& 0\\
					\left(\dot\omega \wedge \Omega\right)_3^3 =& \left(a_3^1 d\theta^1\right) \wedge \left({R_{131}}^3 d\theta^1 \wedge d\theta^3\right)\\
					=&0
				\end{align*} and so $\tr\left(\dot\omega \wedge \Omega\right) = 0$.

				For $n=2$, the same computations as above show that $\Omega$ is of the form
				\begin{equation*}
					\left( \begin{array}{ccc}
							0 & - d\theta^1 \wedge d\theta^2 & \scriptstyle{ {R_{131}}^3 d\theta^1 \wedge d\theta^3 + {R_{231}}^3 d\theta^2 \wedge d\theta^3}\\
							\sin^2\theta^1 d\theta^1 \wedge d\theta^2 & 0 & \scriptstyle{ {R_{132}}^3 d\theta^1 \wedge d\theta^3 + {R_{232}}^3 d\theta^2 \wedge d\theta^3}\\
							\scriptstyle{ {R_{133}}^1 d\theta^1 \wedge d\theta^3 + {R_{233}}^1 d\theta^2 \wedge d\theta^3} & \scriptstyle{ {R_{133}}^2 d\theta^1 \wedge d\theta^3 + {R_{233}}^2  d\theta^2 \wedge d\theta^3 }& 0
						\end{array}\right).
				\end{equation*}

				The derivative of the connection form is 
				\begin{equation*}
					\dot\omega = \left(\begin{array}{ccc}
							a_1^1 d\theta^1 + b_1^1 d\theta^2 & a_1^2 d\theta^1 +  b_1^2 d\theta^2 & c_1^3 d\theta^3\\
							a_2^1 d\theta^1  + b_2^1 d\theta^2 & a_2^2 d\theta^1 + b_2^2 d\theta^2 & c_2^3 d\theta^3\\
							c_3^1 d\theta^3 & c_3^2 d\theta^3 & a_3^3 d\theta^1 + b_3^3 d\theta^2
						\end{array}\right).
				\end{equation*}

				As in the case $n=1$, computation shows that both $\tr \dot{\omega} \tr \Omega=0$ and $\tr \left( \dot{\omega}\wedge \Omega \right)=0$. Therefore $P_1\left( \dot{\omega}\wedge\Omega \right)=0$ and $\dfrac{d}{ds}TP\left( \omega\left( s \right) \right)\mid_{s=0} = exact$.
			\end{proof}

			This lemma combined with the fact that the Ricci flow preserves isometries yields the following theorem.

			\begin{thm}
				If $m+n=3$ and $f:S^n\to \R^+$, then for the manifold $M = S^n \times_f S^m$ and the first Pontryagin polynomial $P_1 \in I^2\left( Gl_{3}\left( \R \right) \right)$ the cohomology class $\left\{ \widetilde{TP\left( \omega \right)} \right\}\in H^3 \left( M,\R/\Z \right)$ is invariant under the Ricci flow.
			\end{thm}

			\begin{proof}
				Consider the projection $\pi : S^n \times_f S^m \to S^n$. For each $x\in S^n$, the fiber $\pi^{-1}(x)$ is going to be an $m$-sphere with radius determined by $f(x)$. Thus the fibers will be isometric to each other up to a factor determined by $f$. At any time $t$ the, the fiber above a point $x$ is going to have isometry group $SO\left( m \right)$ with isotropy subgroup $SO\left( m-1 \right)$. Thus the fiber will be isometric to a sphere of some radius determined by $t$ and $x\in S^n$.

				If $n=1$, any metric on $S^1$ is clearly conformally equivalent to the round metric. For $n=2$, the uniformization theorem says that any metric on $S^2$ is conformally equivalent to the round metric.  Multiplying the metric by the conformal factor will yield a metric that is a warped product of $S^n$ with $S^{m}$.

				By lemma \ref{lem:0deriv} it follows that the derivative of the form $TP_1\left( \omega(t) \right)$ at any time is exact. Thus the class $\left\{ \widetilde{TP_1\left( \omega \right)} \right\} \in H^{2l-1}\left( E(M), \R/\Z \right)$ is invariant. Therefore the class $\left\{ \widetilde{TP_1\left( \omega \right)} \right\}\in H^3\left( M,\R/\Z \right)$ must also be invariant.
			\end{proof}

			\section{Generalized Berger Spheres}

			Consider the manifold $S^3$ identified with the Lie group $SU(2)$. The identification is from the map that sends $(x,y) \in S^3\subset C^2$  to $\left(\begin{array}{cc} y & x\\ -\bar x & \bar y\end{array}\right)\in SU(2)$. The following is the basis for the Lie algebra $\mathfrak{su}(2)$  of $SU(2)$.
			\begin{equation}
				X_1 = \left(\begin{array}{cc} i & 0 \\ 0 & -i\end{array}\right), X_2=\left(\begin{array}{cc} 0 & 1 \\ -1 & 0\end{array}\right), X_3 = \left(\begin{array}{cc} 0 & i \\ i & 0\end{array}\right).
			\end{equation}

			The basis $\left< X_1,X_2,X_3\right>$ forms an orthonormal basis on $S^3$ under the standard metric. The Lie bracket of these vector fields is $\left[X_i, X_{i+1}\right] = 2X_{i+2}$ where the indices are taken mod 3. 

			\begin{defn}
				\label{dfn:genberger}
				Define a new metric on $S^{3}$, $\bar{g}$, such that $\left\langle \lambda_1^{-1}X_{1},\lambda_2^{-1}X_{2},\lambda_3^{-1}X_{3}\right\rangle $ is an orthonormal basis where $\lambda_{i}$ are constants. For notation, let $\bX_{i}=\lambda_i^{-1}X_{i}$. Additionally, let $\left< \bt^1,\bt^2,\bt^3 \right>$ be the orthonormal coframe. The manifold $S^3$ equipped with this metric is called the \emph{generalized Berger sphere}. If $ \lambda_2 = \lambda_3 = 1$ then this is the standard \emph{Berger collapsed sphere} or \emph{Berger  sphere}.
			\end{defn}

			In this section $i,j,k$ are fixed and distinct from one another, so the Einstein summation convention does \emph{not} apply to these indices. Sums will use the letters $p$ and $q$, which do follow the Einstein summation convention. The terms $\omega_i^j(t)$, $\Omega_i^j(t)$, $R_{ij}(t)$, and $\dot{\omega}_i^j(t)$ will represent the connection 1-forms, curvature 2-forms, Ricci curvature, and the derivative of connection 1-forms, respectively, at time $t$. Finally, $\epsilon_{ijk}$ represents the sign of the permutation $(i j k)$.

			Consider the manifold $M$ which is a generalized Berger sphere with basis and coframe as defined in definition~\ref{dfn:genberger} . Under the Ricci flow, $g_t$ will represent the metric at time $t$ with $g_0 = \overline{g}$. The frame, $\left<\bX_1,\bX_2,\bX_3\right>$, and coframe, $\left<\bt^1,\bt^2,\bt^3\right>$, are not being evolved with time and depend only on the initial parameters, $\lambda_i$.

			Using the Koszul formula, the connection can be computed to be $\nabla_{\bX_i}\bX_j = \varepsilon_{ijk}\frac{-\lambda_i^2 + \lambda_j^2 + \lambda_k^2}{\lambda_i\lambda_j\lambda_k} \bX_k$ and $\nabla_{\bX_i}\bX_i = 0$.
			The connection and curvature forms at time 0 are
			\begin{align*}
				\omega_i^j(0) =& \varepsilon_{kij}\frac{-\lambda_k^2 + \lambda_i^2 + \lambda_j^2}{\lambda_k\lambda_i\lambda_j}\bt^k\\
				\Omega_i^j(0) =& \frac{3\lambda_k^4 - \lambda_i^4-\lambda_j^4 + 2\lambda_i^2\lambda_j^2 - 2\lambda_i^2 \lambda_k^2 - 2\lambda_j^2 \lambda_k^2}{\lambda_i^2\lambda_j^2\lambda_k^2}\bt^i \wedge \bt^j.  
			\end{align*} (For details on these and the following computations, see Appendix \ref{app:computations}).

			The Ricci curvature terms at time 0 are
			\begin{equation*}
				R_{ii}(0) = \frac{2\lambda_i^4 - 2\lambda_k^4 -2\lambda_j^4 + 4\lambda_j^2 \lambda_k^2}{\lambda_i^2 \lambda_j^2 \lambda_k^2}.
			\end{equation*}

			The derivative of the connection forms can be computed directly.
			\begin{align*}
				\partial_{t}\left\langle \nabla_{\bX_{k}}\bX_{i},\bX_{j}\right\rangle = & \partial_{t}\left\langle \omega_{i}^{p}\left(\bX_{k}\right)\bX_{p},\bX_{j}\right\rangle \\
				= & \dot{\omega}_{i}^{p}\left(\bX_{k}\right)\left\langle \bX_{p},\bX_{j}\right\rangle +\omega_{i}^{p}\left(\bX_{k}\right)\partial_{t}\left\langle \bX_{p},\bX_{j}\right\rangle \\
				= & \dot{\omega}_{i}^{j}\left(\bX_{k}\right)\left\langle \bX_{j},\bar{X_{j}}\right\rangle +\omega_{i}^{p}\left(\bX_{k}\right)\partial_{t}\left\langle \bX_{p},\bX_{j}\right\rangle \\
				= & \dot{\omega}_{i}^{j}\left(\bX_{k}\right)-2\omega_{i}^{p}\left(\bX_{k}\right)R_{pj}\\
				= & \dot{\omega}_{i}^{j}\left(\bX_{k}\right)-2\omega_{i}^{j}\left(\bX_{k}\right)R_{jj},
			\end{align*}
			so that
			\begin{equation*}
				\dot{\omega}_{i}^{j}(0)\left(\bX_{k}\right)=\partial_{t}\left\langle \nabla_{\bX_{k}}\bX_{i},\bX_{j}\right\rangle +2\omega_{i}^{j}\left(\bX_{k}\right)R_{jj}.
			\end{equation*}
			At time 0, computation yields

			\begin{equation}
				\label{eqn:omdij}
				\dot{\omega}_{i}^{j}(0) = -2 \varepsilon_{kij}\frac{-\lambda_k^2 + \lambda_i^2 + \lambda_j^2}{\lambda_k\lambda_i\lambda_j} \left( R_{11}+R_{22}+R_{33} \right)\bt^k
			\end{equation}

			Since $\Omega$ is skew-symmetric, the first Pontryagin polynomial reduces to 
			\begin{equation*}
				P_1\left( A\otimes B\right) = -\frac{1}{2\pi^2}\tr\left( AB \right).
			\end{equation*}Thus the derivative at time 0 of $TP_1\left( \omega \right)$ is $-\frac{1}{2\pi^2}\tr\left( \dot{\omega}\wedge\Omega \right) + exact$.

			\begin{lem}
				\label{lem:berger}
				For a Berger sphere, the Chern-Simons class, $\left\{ \widetilde{TP_1\left( \omega \right)} \right\}$, is invariant if and only if $\lambda=1$.
			\end{lem}

			\begin{proof}
				The derivative of $TP_1\left( \omega \right)$ at time 0 is
				\begin{equation}
					\label{berger}
					\dfrac{d}{dt} TP\left( \omega(t) \right) \mid_{t=0} = -\frac{2}{\pi^2}\lambda\left( \lambda^2-1 \right)^2 \theta^1 \wedge \theta^2 \wedge \theta^3 + exact.
				\end{equation}

				If $\lambda=1$, the Berger sphere is just a 3-sphere. It is already known that the Ricci flow conformally shrinks a 3-sphere to a point \cite[Chapter 1.6]{Chow2004}  For any other value of $\lambda$, however, the Chern-Simons form changes by something that is not exact and therefore the Chern-Simons class $\left\{ TP_1\left( \omega \right) \right\}\in H^3\left( E(M),\R \right)$ varies.  Since the class varies continuously, its reduction mod $\Z$ must also vary and therefore $\left\{ \widetilde{TP\left( \omega \right)} \right\}\in H^{3}\left( M,\R/\Z \right)$  must vary.
			\end{proof}

			The derivative of the Chern-Simons form simplifies to
			\begin{equation}
				\label{3consts}
				\textstyle{\partial_{t}TP_{1}\left(\omega\right)\mid_{t=0} =\frac{16}{\pi^2\lambda_{1}^{5}\lambda_{2}^{5}\lambda_{3}^{5}}\left(\sum\limits_p\lambda_{p}^{10}-\sum\limits_{p\neq q}\lambda_{p}^{8}\lambda_{q}^{2}+\sum\limits_{p}\lambda_{p}^{6}\lambda_{p+1}^{2}\lambda_{p+2}^{2}\right)\bt^{1}\wedge\bt^{2}\wedge\bt^{3}}.
			\end{equation}

			\begin{thm}
				\label{thm:3const}
				For the Berger sphere generalized by three constants, the Chern-Simons class $\left\{ \widetilde{TP_1\left( \omega \right)} \right\}\in H^3\left( M,\R/\Z \right)$ defined by the first Pontryagin polynomial form is invariant only if $\lambda_1=\lambda_2=\lambda_3$.
			\end{thm}

			\begin{proof}
				If $\lambda_1=\lambda_2=\lambda_3$, the generalized Berger sphere is just a 3-sphere of a different radius.  This changes conformally and the Chern-Simons forms are conformal invariants. If exactly two of the constants are the same, then this the manifold is conformally equivalent to a Berger sphere, which was dealt with in corollary \ref{lem:berger}.

				Let $g$ be the metric such that $\left< \lambda_1^{-1} X_1, \lambda_2^{-1}X_2,\lambda_3^{-1}X_3 \right>$ is an orthonormal basis of vector fields.  Without loss of generality, assume that $\lambda_1^{-1} < \lambda_2^{-1} < \lambda_3^{-1}$. This metric is conformally related to the metric, $\hat{g}$, such that $\left< \lambda_3\lambda_1^{-1}X_1, \lambda_3\lambda_2^{-1}X_2,X_3 \right>$ is an orthonormal basis. 

				Let $\alpha=\lambda_3^{-1}\lambda_1$ and $\beta = \lambda_3^{-1}\lambda_2$. It is clear that $\alpha >1$ and $\beta >1$. And so the numerator of the coefficient of \ref{3consts} reduces to
				\begin{equation}
					F\left(\alpha,\beta\right) = \alpha^{10} + \beta^{10} - \alpha^8 \beta^2 - \beta^8\alpha^2  - \alpha^8 - \beta^8 + \alpha^6 \beta^2 + \beta^6 \alpha^2 + \alpha^2 \beta^2 - \alpha^2 - \beta^2 +1.
				\end{equation}
				This can be factored into either of the following.
				\begin{align}
					& \left( \alpha^2-\beta^2 \right)^2\left( \left( \alpha^2-1 \right)\left( \alpha^4 + \alpha^2\beta^2 + \beta^4 \right) + \beta^6 \right) + \left( \alpha^2-1 \right)\left( \beta^2-1 \right)\\
					& \left( \alpha^2-\beta^2 \right)^2\left( \left( \beta^2-1 \right)\left( \beta^4 + \beta^2\alpha^2 + \alpha^4 \right) + \alpha^6 \right) + \left( \alpha^2-1 \right)\left( \beta^2-1 \right).
				\end{align}
				Since $\alpha > 1$ and $\beta > 1$, this coefficient is clearly positive.

				Thus, $F\left( \alpha,\beta \right) = 0 $ if and only if $\alpha = \beta = 1$. Otherwise the class $\left\{ TP_1\left( \omega \right) \right\}\in H^3\left( E(M),\R \right)$ varies continuously and so does its reduction mod $\Z$. Therefore $\left\{ \widetilde{TP_1\left( \omega \right)} \right\}\in H^3\left( M, \R/\Z \right)$ varies.
			\end{proof}

			\appendix
			\section{Computations for Section 3}
			\label{app:computations}

			As in section 3, $i,j,k$ are fixed and distinct. Sums in this section will use the letters $p$ and $q$, which do follow the Einstein summation convention. All computations here occur at time 0.

			Since $\left[\bX_i,\bX_j\right] = 2\epsilon_{ijk} \lambda_i^{-1}\lambda_j^{-1} \bX_k$, the Koszul formula gives
			\begin{align*}
				\nab{\bX_i}{\bX_j}{\bX_k} =& \frac12 \koszr{\bX_i}{\bX_j}{\bX_k}\\
				=& \frac{\epsilon_{ijk}}{2} \left(2\lambda_i^{-1}\lambda_j^{-1} \left<X_k,\bX_k\right> - 2\lambda_j^{-1}\lambda_k^{-1} \left<X_i,\bX_i\right> + 2\lambda_k^{-1}\lambda_i^{-1}\left<X_j,\bX_j\right>\right)\\
				=& \epsilon_{ijk}\left(\frac{\lambda_k}{\lambda_i\lambda_j} - \frac{\lambda_i}{\lambda_j\lambda_k} + \frac{\lambda_j}{\lambda_i\lambda_k}\right)\\
				=& \epsilon_{ijk}\left(\frac{-\lambda_i^2+\lambda_j^2 +\lambda_k^2}{\lambda_i\lambda_j\lambda_k}\right).
			\end{align*}

			Thus
			\begin{align*}
				\nabla_{\bX_i} \bX_i =& 0.\\
				\nabla_{\bX_i} \bX_j =& \epsilon_{ijk}\left(\frac{-\lambda_i^2+\lambda_j^2 +\lambda_k^2}{\lambda_i\lambda_j\lambda_k}\right)\bX_k.
			\end{align*}

			The connection form us computed using $\na{k}\bX_i = \omega_i^p \left(\bX_k\right)\bX_p$. 
			\begin{equation*}
				\omega_i^j\left(\bX_k\right) = \epsilon_{kij}\left( \frac{-\lambda_k^2 + \lambda_i^2 + \lambda_j^2}{\lambda_i\lambda_j\lambda_k} \right).
			\end{equation*} Thus
			\begin{equation*}
				\omega_i^j = \epsilon_{ijk} \left(\frac{\lambda_i^2 + \lambda_j^2 - \lambda_k^2}{\lambda_i\lambda_j\lambda_k}\right) \bt^k
			\end{equation*}

			Using the formula $R\left(\bX_i,\bX_j\right)\bX_k = \Rm{\bX_i}{\bX_j}{\bX_k}$ the Riemannian curvature tensor is calculated as

			\begin{align*}
				R\left(\bX_i,\bX_j\right)\bX_i =& \Rm{\bX_i}{\bX_j}{\bX_i}\\
				=& \epsilon_{jik}\left(\frac{-\lambda_j^2 + \lambda_i^2 + \lambda_k^2}{\lambda_i\lambda_j\lambda_k}\right) \na{i}\bX_k -2\epsilon_{ijk}\frac{\lambda_k}{\lambda_i\lambda_j} \na{k}\bX_i\\
				=& \epsilon_{jik}\epsilon_{ikj}\left(\frac{-\lambda_j^2 + \lambda_i^2 + \lambda_k^2}{\lambda_i\lambda_j\lambda_k}\right)\left(\frac{-\lambda_i^2 + \lambda_k^2 + \lambda_j^2}{\lambda_i\lambda_j\lambda_k}\right)\bX_j\\
				&- 2\epsilon_{ijk}\epsilon_{kij}\frac{\lambda_k}{\lambda_i\lambda_j}\left(\frac{-\lambda_k^2 + \lambda_i^2 + \lambda_j^2}{\lambda_i\lambda_j\lambda_k}\right) \bX_j\\
				=& \frac{(-\lambda_j^2+\lambda_i^2\lambda_k^2)(-\lambda_i^2+\lambda_k^2+\lambda_j^2)-2\lambda_k^2(-\lambda_k^2+\lambda_i^2+\lambda_j^2)}{\lambda_i^2\lambda_j^2\lambda_k^2} \bX_j
			\end{align*}

			The other computations are similar. Combining these with the formula $R_{ii}= {R_{pii}}^p$, the Ricci curvature terms are

			\begin{equation*}
				R_{ii} = 2\left(\frac{\lambda_i^4 -\lambda_j^4 - \lambda_k^4  + 2\lambda_j^2\lambda_k^2}{\lambda_i^2\lambda_j^2\lambda_k^2}\right)
			\end{equation*}

			The curvature form, $\Omega$, is computed using $\Omega_i^j =\frac12 {R_{pqi}}^j \bt^p\wedge\bt^q$ and is
			\begin{equation*}
				\Omega_i^j = \frac{\lambda_i^4 +\lambda_j^4 - 3\lambda_k^4 - 2\lambda_i^2\lambda_j^2  + 2\lambda_i^2\lambda_k^2+2\lambda_j^2\lambda_k^2}{\lambda_i^2\lambda_j^2\lambda_k^2}\bt^i\wedge\bt^j
			\end{equation*}

			\bibliography{paper001}{}
			\bibliographystyle{amsplain}
		\end{document}